\documentclass[11pt,reqno]{amsart}
\usepackage{mathrsfs}
\usepackage{amssymb}
\usepackage{amsmath}
\usepackage{amscd}
\usepackage{verbatim}
\usepackage{arydshln}
\usepackage[colorlinks, linkcolor=blue, anchorcolor=blue, citecolor=blue]{hyperref}
\begin{document}
\setlength{\oddsidemargin}{0cm} \setlength{\evensidemargin}{0cm}

\theoremstyle{plain} \makeatletter
\newtheorem{theorem}{Theorem}[section]
\newtheorem{proposition}[theorem]{Proposition}
\newtheorem{lemma}[theorem]{Lemma}
\newtheorem{corollary}[theorem]{Corollary}
\newtheorem{coro}[theorem]{Corollary}

\theoremstyle{definition}
\newtheorem{defi}[theorem]{Definition}
\newtheorem{notation}[theorem]{Notation}
\newtheorem{exam}[theorem]{Example}
\newtheorem{prop}[theorem]{Proposition}
\newtheorem{conj}[theorem]{Conjecture}
\newtheorem{prob}[theorem]{Problem}
\newtheorem{remark}[theorem]{Remark}
\newtheorem{claim}[theorem]{Claim}

\newcommand{\SO}{{\mathrm S}{\mathrm O}}
\newcommand{\SU}{{\mathrm S}{\mathrm U}}
\newcommand{\Sp}{{\mathrm S}{\mathrm p}}
\newcommand{\so}{{\mathfrak s}{\mathfrak o}}
\newcommand{\Ad}{{\mathrm A}{\mathrm d}}
\newcommand{\ad}{{\mathrm a}{\mathrm d}}
\newcommand{\m}{{\mathfrak m}}
\newcommand{\g}{{\mathfrak g}}
\newcommand{\h}{{\mathfrak h}}
\newcommand{\n}{{\mathfrak n}}

\numberwithin{equation}{section}

\title{On the Geodesic Orbit Property for Lorentz Manifolds}
\author{Zhiqi Chen, Joseph A. Wolf and Shaoxiang Zhang}

\address[Zhiqi Chen]{School of Mathematics and Statistics,
Guangdong University of Technology, Guangzhou, 510520, P.R. China. {\em email:}{ \tt
chenzhiqi@nankai.edu.cn}}

\address[Joseph A. Wolf]{Department of Mathematics, University of California, 
Berkeley, CA 94720-3840, U. S. A. {\em email:}{ \tt jawolf@math.berkeley.edu}}

\address[Shaoxiang Zhang]{College of Mathematics and System Science, Shandong University of Science and Technology, Qingdao 266590, P.R. China.  Corresponding author.
{\em email:}{ \tt zhangshaoxiang@mail.nankai.edu.cn}}

\date{}
\begin{abstract}
The geodesic orbit property has been studied intensively for Riemannian
manifolds.  Geodesic orbit spaces are homogeneous and allow simplifications 
of many structural questions using the Lie algebra of the isometry group.
Weakly symmetric Riemannian manifolds are geodesic orbit spaces.  Here we define
the property ``naturally reductive'' for pseudo--Riemannian manifolds and 
note that those manifolds are geodesic orbit spaces.  A few years ago two of the authors proved
that weakly symmetric pseudo--Riemannian manifolds are geodesic orbit spaces.
In particular that result applies to pseudo--Riemannian Lorentz manifolds.  
Our main results are Theorems \ref{th41} and \ref{th42}.  In the
Riemannian case the nilpotent isometry group for a geodesic orbit nilmanifold 
is abelian or $2$--step nilpotent.  Examples show
that this fails dramatically in the pseudo--Riemannian case.  Here we 
concentrate on the geodesic orbit property for Lorentz nilmanifolds
$G/H$ with $G = N \rtimes H$ and $N$ nilpotent.   When the metric is 
nondegenerate on $[\n,\n]$, Theorem \ref{th41} shows that $N$ either is at
most $2$--step nilpotent as in the Riemannian situation, or is $4$--step
nilpotent, but cannot be $3$--step nilpotent.  Examples show that
these bounds are the best possible.  Surprisingly, Theorem \ref{th42}
shows that $N$ is at
most $2$--step nilpotent when the metric is degenerate on $[\n,\n]$.
Both theorems give additional structural information and specialize
to naturally reductive and to weakly symmetric Lorentz nilmanifolds.
\end{abstract}

\maketitle

\section{Introduction}\label{sec1}

A Riemannian manifold is called a {\em geodesic orbit space} if every
geodesic is the orbit of a one parameter group of isometries.  They are 
homogeneous, thus amenable to study by Lie algebra methods, and have been
studied intensively by many mathematicians.  Two particular cases are
of special interest are {\em naturally reductive spaces} and {\em
weakly symmetric spaces}.  Here we study extension of results on these
topics from Riemannian manifolds to pseudo--Riemannian manifolds, in 
particular to Lorentz manifolds.  Some of the results, especially
for geodesic orbit Lorentz manifolds, are quite surprising.

A connected Riemannian homogeneous space $M = G/H$ is a {\em nilmanifold} 
if the isometry group $G$ has a connected nilpotent subgroup $N$ that is
transitive on $M$.  In that case $N$ is the nilradical of $G$, and
$G$ is the semidirect product $N\rtimes H$.  If $M$ is a geodesic orbit
Riemannian manifold then $N$ is commutative or $2$--step nilpotent.  This fails 
dramatically for connected pseudo--Riemannian geodesic orbit nilmanifolds. 
Our main results form a very strong extension of the $2$--step nilpotent 
theorem from Riemannian geodesic orbit spaces to Lorentz geodesic orbit
spaces.  Of course these results hold for special cases such as
naturally reductive spaces and, more importantly, weakly symmetric spaces.

Consider a geodesic orbit Lorentz nilmanifold $M = G/H$ with $G = N \rtimes H$ 
and $N$ nilpotent.  Then $N$ is a closed normal subgroup of $G$ and is 
simply transitive on $M$, so we can view $M$ as the group $N$ with a left
invariant $\Ad(H)$--invariant Lorentz metric.  

Our first main result, 
Theorem \ref{th41}, applies to the case where the metric is nondegenerate
on $[\n,\n]$.  It says that $N$ either is at
most $2$--step nilpotent as in the Riemannian situation, or is $4$--step
nilpotent, but cannot be $3$--step nilpotent.  Examples show that
these bounds are the best possible.  
Our second main result, Theorem \ref{th42}, shows that $N$ is at
most $2$--step nilpotent when the metric is degenerate on $[\n,\n]$.
As one expects, both theorems give additional structural information and 
specialize to naturally reductive and to weakly symmetric Lorentz nilmanifolds.

In Section \ref{sec2} we indicate the background on geodesic orbit spaces. 
We describe the particular cases of naturally reductive spaces and weakly 
symmetric spaces.  The emphasis is on extending definitions and results 
from the Riemannian case to the pseudo--Riemannian cases.  Weakly symmetric
spaces are the most important special case (see \cite{JA}), so in
Section \ref{sec3} we go into more detail on that.  Sections \ref{sec4}
and \ref{sec5} contain our main results and their proofs.

\section{Geodesics in Pseudo-Riemannian Manifolds}
\label{sec2}

In this section we discuss the geodesic orbit property for Riemannian 
manifolds and extensions to the pseudo--Riemannian setting.
We will need that in the proofs of our main results, Theorems \ref{th41} and
\ref{th42}.  
First we recall some background and some results of the first two authors 
from \cite{CW}.

\begin{defi}\label{de41}
A pseudo-Riemannian manifold $M$ is called a \textit{geodesic orbit space}
if every maximal geodesic in $M$ is an orbit of a one-parameter group of
isometries of $M$. \hfill $\diamondsuit$
\end{defi}

Pseudo-Riemannian geodesic orbit spaces are geodesically complete
and homogeneous \cite{CW}.  Let $G$ be a transitive Lie group of isometries
of $M$, say $M = G/H$ with base point $x_0 = 1H$.  We say that a nonzero 
element $\xi$ in The Lie algebra $\mathfrak g$ is a {\em geodesic vector} 
if the image of
$t \mapsto \gamma(\varphi(t)) = \exp(t\xi)x_0$ is a maximal geodesic, 
where $\varphi$ is a diffeomorphism of $\mathbb R$ onto an open interval 
in $\mathbb R$.  Then $\gamma$ is a {\em homogeneous geodesic}.
In the Riemannian case $\varphi(t) = t$ and $t$ is the affine parameter.  

Suppose that we have a
reductive decomposition $\mathfrak g = \mathfrak m + \mathfrak h$ 
(vector space direct sum)
with $[\mathfrak h, \mathfrak m] \subset \mathfrak m$, which of course is
automatic in the Riemannian case where $H$ is compact.  
As usual we write $\pi_\h$
and $\pi_\m$ for projections to the summands, and $\xi_\h$ and $\xi_\m$
for the components of an element $\xi \in \g$.  
Then $M$ is a geodesic orbit space
({\em relative to $\mathfrak g$ and $\mathfrak g = \mathfrak h + \mathfrak m$})
if and only if, for every $\xi \in \mathfrak m$, there is an
$\eta \in \mathfrak h$ such that $\xi + \eta$ is a geodesic vector.
In other words, a homogeneous pseudo-Riemannian manifold $M$ is a 
geodesic orbit space if every geodesic of $M$ is homogeneous.

\begin{proposition}{\rm (\bf Geodesic Lemma)}.
Let $M=G/H$ be a pseudo-Riemannian homogeneous space.  Suppose that there
is a reductive decomposition $\g = \m + \h$.  Let $x_0 = 1H \in G/H$ and
$\xi \in \g$.  Then the curve $\gamma(t)= \exp(t\xi)(x_0)$ is a geodesic curve 
with respect to some parameter $s$ if and only if
\begin{equation}\label{go-lemma}
\langle [\xi, \zeta]_\m, \xi_\m \rangle=k\langle \xi_\m, \zeta\rangle
\end{equation}
for all $\zeta \in \m$, where $k$ is a constant.  If $k=0$, then $t$ is an
affine parameter for $\gamma$.  If $k \ne 0$, then $s= e^{-kt}$ is an
affine parameter for $\gamma$, and $\gamma$ is a null curve in $M$.
\end{proposition}

\begin{corollary}\label{cor-geod-lem}
Given a reductive decomposition $\g = \m + \h$, 
the geodesic orbit property $M$ is equivalent to the following condition.
If $\xi \in \m$ there exist $\alpha \in \h$ and a constant $k$
such that, if $\zeta \in \m$ then $\langle [\xi + \alpha, \zeta]_\m,\xi\rangle 
= k\langle \zeta, \xi \rangle$.
\end{corollary}
\smallskip

For the notion of homogeneous geodesic and the formula in the Geodesic Lemma 
see
\cite{KV} for the Riemannian case, and then \cite{DK}, \cite{JPS} and 
\cite{SP} for the pseudo--Riemannian case.

\centerline{-----o-----}
\smallskip

A homogeneous Riemannian manifold $M=G/H$, with 
reductive decomposition $\g = \m + \h$, is {\em naturally reductive}
(with respect to $G$) if $\pi_{\mathfrak m}\cdot\ad(\xi)|_{\mathfrak m}$
is skew symmetric for every $\xi \in \mathfrak{m}$.  In terms of the
inner product on $\mathfrak{m}$ this condition is
$\langle [\xi,\eta]_\mathfrak{m},\zeta\rangle 
	+ \langle \eta, [\xi,\zeta]_\mathfrak{m}\rangle = 0$
for all $\xi, \eta, \zeta \in \mathfrak{m}$.  The case $\xi = \eta$ is:
$\xi,\, \zeta \in \mathfrak{m} \Rightarrow 
	\langle \xi, [\xi,\zeta]_\mathfrak{m}\rangle = 0$.
As noted in \cite[Proposition 1.7(a)]{G},
following \cite{KV} this says that every $\xi \in \mathfrak{m}$ is a 
homogeneous vector.  However we have a simpler treatment that avoids 
\cite{G} and \cite{KV},
and is valid for the pseudo--Riemannian case as well.

\begin{defi}\label{def0}
Let $M=G/H$ be a pseudo-Riemannian homogeneous space.  Suppose that there
is a reductive decomposition $\g = \m + \h$.  Let $x_0 = 1H \in G/H$ and
$\xi \in \g$ and let $L$ be the group of linear transformations of $\m$
that preserve the inner product.  If $\ad(\xi)|_\m$ belongs to the Lie
algebra of $L$ for every $\xi \in \m$, then $M$ is {\bf naturally reductive}
(with respect to $G$ and the reductive decomposition 
$\g = \m + \h$). \hfill $\diamondsuit$
\end{defi}
\begin{remark} The naturally reductive property depends on the
choice of transitive isometry group and the reductive decomposition,
even in the Riemannian case.  See \cite{NN}.  We thank Yurii 
Nikonorov for that and other references.
\hfill $\diamondsuit$
\end{remark}

As in the Riemannian case, the defining condition for ``naturally 
reductive'' in terms of the inner product and the Lie algebra $\mathfrak l$
of $L$ is
$\ad(\eta)|_\m \in \mathfrak l$ for all $\eta \in \m$, i.e. 
$\langle [\xi,\zeta]_\mathfrak{m},\eta\rangle 
        + \langle \zeta, [\xi,\eta]_\mathfrak{m}\rangle = 0$
for all $\xi, \eta, \zeta \in \mathfrak{m}$.  The case $\xi = \eta$ 
says $\langle [\xi,\zeta]_\m , \xi \rangle = 0$ for all 
$\zeta,\,\xi \in \m$.
In other words if $\xi, \zeta \in \m$ then \ref{go-lemma} holds with $k = 0$.
We have proved

\begin{proposition}\label{propnat} {\rm \cite[pp. 200--202]{KN}}
Naturally reductive homogeneous pseudo--Riemannian manifolds are 
pseudo--Riemannian geodesic orbit spaces.
\end{proposition}

\centerline{-----o-----}
\smallskip

Selberg introduced an important extension of the class of Riemannian 
symmetric spaces: Riemannian weakly symmetric spaces \cite{AS}.  
As mentioned in the Introduction, there are a number of characterizations
equivalent to the definition of weak symmetry for a Riemannian manifold.
The paper \cite{CW} of Chen and Wolf uses the
characterization that is geometrically most accessible.

\begin{defi}\label{def1}
Let $(M, g)$ be a pseudo-Riemannian manifold. Suppose that for every
$x \in M$ and every nonzero tangent vector $\xi \in T_xM$, there is an
isometry $\phi = \phi_{x,\xi}$ of $M$ such that $\phi(x) = x$ and
$d\phi(\xi) = -\xi$. Then $(M, g)$ is a \textit {weakly symmetric
pseudo-Riemannian manifold}. In particular, if $\phi_{x,\xi}$ is
independent of $\xi$ then $M$ is symmetric. \hfill $\diamondsuit$
\end{defi}

Riemannian symmetric spaces 
are geodesic orbit spaces.  A few years ago the first two authors 
extended this result to weak symmetry and indefinite metric:

\begin{prop}\label{prop41} {\rm \cite[\S 4]{CW}}
Weakly symmetric pseudo-Riemannian manifolds are geodesic orbit spaces.
\end{prop}

In other words, if $(M,g)$ is a weakly symmetric pseudo--Riemannian 
manifold, and if $\gamma$ is a maximal geodesic and $x \in \gamma$, there
is an isometry of $M$ which is a non-trivial involution on $\gamma$ with
$x$ as fixed point.  

\centerline{-----o-----}
\smallskip

There are many other classes of Riemannian manifolds related to Riemannian
geodesic orbit spaces, for example normal homogeneous spaces, D'Atri 
spaces and Damek-Ricci spaces.  See the survey article \cite{KPV} and the 
references there for a discussion of these classes and their relation to 
weakly symmetric spaces and geodesic orbit spaces, \cite{WZ} and \cite{JA}
for a number of examples, and the comprehensive book \cite{BN} on 
Riemannian geodesic orbit spaces.  Rather than digress to consider these 
various classes in any depth, in Section \ref{sec3} 
below we only sketch
some important background material applicable to weakly symmetric 
pseudo--Riemannian manifolds.

\section{Weakly Symmetric Pseudo-Riemannian Manifolds} \label{sec3}
In this section we sketch some of the main results on weakly symmetric
pseudo--Riemannian manifolds.  A. Selberg's original definition of weakly 
symmetric space holds also for pseudo-Riemannian manifolds, but it is 
rather complicated\footnote{Selberg: Let $(M, g)$ be a Riemannian 
manifold. If
there exists a subgroup $G$ of the isometry group $I(M,g)$ of $M$ acting
transitively on $M$ and an involutive isometry $\mu$ of $(M, g)$ with
$\mu G = G \mu$ such that whenever $x, y \in M$ there exists $\phi \in G$
with $\phi (x) = \mu(y)$ and $\phi (y) = \mu(x)$, then $(M, g)$ is a
\textit{weakly symmetric Riemannian manifold}.} 
so we use Definition \ref{def1} above.
\smallskip

A connected homogeneous pseudo-Riemannian manifold need not be geodesically
convex, but any two points can be joined by a broken geodesic. 
Going segment by segment along broken geodesics, as in the Riemannian case
we have

\begin{prop} \label{prop1}
A pseudo-Riemannian manifold $(M, g)$ is weakly symmetric
if and only if for any two points $x$, $y$ $\in M$ there is an isometry of
$M$ mapping $x$ to $y$ and $y$ to $x$.
\end{prop}

Recall the De Rham-Wu decomposition theorem \cite{HW}.  Let $(M, g)$ be a
complete simply connected pseudo-Riemannian manifold, $x \in M$, and
$T_x(M) = T_{x,0} \oplus \cdots \oplus T_{x,r}$
a decomposition of the tangent space at $x$ into holonomy
invariant mutually orthogonal subspaces, where the holonomy group at $x$ is
trivial on $T_{x,0}$ and irreducible on the other $T_{x,i}$\,.
Suppose that the pseudo-Riemannian metric $g$ has nondegenerate restriction
to $T_{x,i}$ for each index $i$\,.
Then $(M, g)$ is isometric to a pseudo-Riemannian direct product
$(M_0, g_0) \times \cdots \times (M_r, g_r)$, where
$x = (x_0, \dots , x_r)$ and for each index, $(M_i, g_i)$ has
tangent space $T_{x,i}$ at $x_i$\,.  As in the Riemannian case $(M_i, g_i)$
is the maximal integral manifold through $x$ of the distribution
obtained by parallel translating $T_{x,i}$ $M$ and equipped with the
metric $g_i$ induced by $g$.  Thus

\begin{prop}
Let $(M, g)$ be a
complete simply connected pseudo-Riemannian manifold, $x \in M$, and
$T_x(M) = T_{x,0} \oplus \cdots \oplus T_{x,r}$
a decomposition of the tangent space at $x$ into holonomy
invariant mutually orthogonal subspaces, where the holonomy group at $x$ is
trivial on $T_{x,0}$ and irreducible on the other $T_{x,i}$\,.
Suppose that the pseudo-Riemannian metric $g$ has nondegenerate restriction
to $T_{x,i}$ for each index $i$\, and let
$(M,g) = (M_0, g_0) \times \cdots \times (M_r, g_r)$ be the
De Rham-Wu decomposition. Then $(M,g)$ is weakly symmetric if and only if
each of the pseudo-Riemannian manifolds  $(M_i, g_i)$ is weakly symmetric.
\end{prop}

\begin{defi}
Let $G$ be a connected Lie group and $H$ be a closed subgroup. Suppose
that $\sigma$ is an automorphism of $G$ such that $\sigma(p) \in Hp^{-1}H$, 
$\forall p \in G$. Then $G/H$ is called a
\textit{weakly symmetric coset space}, $(G, H)$ is called a
\textit{weakly symmetric pair}, and $\sigma$ is called a
\textit{weak symmetry} of $G/H$. \hfill $\diamondsuit$
\end{defi}

It is easy to see that a weakly symmetric pseudo-Riemannian manifold 
is a weakly symmetric coset space $G/H$ where $G$ is the isometry group.
\smallskip

In the context of homogeneous spaces $G/H$ where $G$ is a connected Lie group
and $H$ is a compact subgroup, one often says that $G/H$ is \textit{commutative}
when the algebra of all $G$-invariant differential operators is commutative.
That is a special case (where $G$ is a connected Lie group)
of the definition in the setting of topological groups: $G$ is a separable
locally compact group, $H$ is a compact subgroup, and the convolution
algebra $L^1(H\backslash G/H)$ is commutative.  Selberg \cite{AS} proved that
a Riemannian weakly symmetric space $M = G/H$ is a commutative space, but
Lauret \cite{JLa} found commutative spaces that are not weakly symmetric.

\section{Lorentz Geodesic Orbit and Weakly Symmetric Nilmanifolds, I}
\label{sec4}
We start with a useful lemma.

\begin{lemma}\cite{AK} \label{lemma4}
Let $B \in \so(n-1, 1)$.  In a suitable basis $\{e_i\}$ of $\mathbb{R}^n$,
either

\begin{equation*}
\begin{aligned}
&\roman{enumi} (1)\, B \text{ is semisimple }, B = \left ( \begin{smallmatrix}
  \mu & 0 & 0 \\ 0 & C & 0 \\ 0 & 0 & -\mu \end{smallmatrix} \right )
\text{ with } C \in \so(n-2) \text{ and } \mu \leq 0 \text{, where }\\
&\langle e_1, e_{n} \rangle=1\,, \langle e_i, e_j \rangle=\delta_{ij}\,
(i, j= 2, \cdots , n-1) \text{ and the other scalar products vanish, or}\\
&\roman{enumi} (2) \, B \text{ is not semisimple },
B = \left ( \begin{smallmatrix} 0 & 1 & 0 & \\ 0 & 0 & 1 & \mathbf{0} \\
  0 & 0 & 0 & \\  & \mathbf{0} &  & C \end{smallmatrix} \right )
\text{ where } C \in \so(n-3),\, \mathbf{0} 
\text{ is the zero matrix of appropriate} \\
&\text{size; } -\langle e_1, e_3 \rangle=\langle e_2, e_2 \rangle=1,
\langle e_i, e_j \rangle=\delta_{ij} (i, j= 4, \cdots , n),
\text{ and the other scalar products vanish.}
\end{aligned}
\end{equation*}

\end{lemma}

There is a typographical error in \cite{AK}; in case $(2)$
there, one should assume $r=e_1$ and $q=e_2$.
\smallskip

From Lemma \ref{lemma4}, any matrix $H \in \so(n-1, 1)$ has $n-2$ purely
imaginary eigenvalues and two non-zero real eigenvalues $\pm \mu$, or
has $n$ purely imaginary eigenvalues, viewing $0$ as purely imaginary.
Now we can state the first of our two main results, the case where
the metric is nondegenerate on $[\n, \n]$.

\begin{theorem}\label{th41}
Let $(M = G/H, \langle\,,\,\rangle)$ be a connected Lorentz geodesic
orbit nilmanifold, where $G = N \rtimes H$ with $N$ nilpotent.  
Then N is abelian, or 2-step nilpotent, or 4-step nilpotent.
\smallskip

Note that $G = N \rtimes H$ defines the reductive decomposition 
$\g=\n + \h$.  Identify $\n$ with the tangent space at $1H$ and
let $\mathfrak{v}$ denote the orthocomplement of $[\n,\n]$ in $\n$.
Suppose that $\langle \cdot, \cdot \rangle$ is nondegenerate on $[\n, \n]$.
Then either $\ad (x)=0$ for any $x \in \mathfrak{v}$, or there is a 
basis $\{x, \tilde x_1,\cdots, \tilde x_s \}$ of
$\mathfrak{v}$ $(s\ge 1)$ such that
\begin{equation*}
\begin{aligned}
&{\ad (x)}|_{[\n, \n]}=
               \left ( \begin{matrix}
                \begin{smallmatrix}
                 0 &1 &0 \\ 0 &0 &1\\ 0 &0 &0 \end{smallmatrix}
                  & {\mathbf 0}\\
                  {\mathbf 0} &
                  {\mathbf 0}
                \end{matrix} \right ) ; \qquad
\ad (\tilde x_1)|_{[\n, \n]}=
              \left(\begin{matrix} {\mathbf 0} &
                \begin{smallmatrix}
               a_1 & a_2&\cdots &a_p\\
               0 &0 &\cdots &0 \\
               0 &0 &\cdots &0 \end{smallmatrix} \\
                \begin{smallmatrix} 0 & 0 & a_1 \\
			\vdots & \vdots & \vdots \\
			0 & 0 & a_p \end{smallmatrix} & {\mathbf 0}
		\end{matrix} \right ); \\
& \ad (\tilde x_i)|_{[\n, \n]}=0 \text{ for } 2\leq s,\, 2 \leq i \leq s; \text{ and }
\ad ([y, z])|_{[\n, \n]}=0 \text{ for all } y, z \in \n.
\end{aligned}
\end{equation*}
\end{theorem}

\begin{proof}
We may assume that $[\n,\n] \ne 0$.
Note that $\mathfrak{v} = [\n, \n]^\perp \cap \n$ is $Ad_G(H)$-invariant,
where of course orthogonality is relative to
$\langle \cdot , \cdot \rangle$.
For any $\eta \in [\n, \n]$, the Geodesic Lemma provides $\alpha_\eta \in \h$ 
such that
$\langle [\eta+\alpha_\eta, \xi]_\n, \eta \rangle=k\langle \eta, \xi \rangle=0$
for every $\xi \in \mathfrak{v}$. It follows that 
$\langle [\xi, \eta]_\n, \eta \rangle=0$. That is, for any 
$\xi \in \mathfrak{v}$ and $\eta,\zeta \in [\n, \n]$,
\begin{equation}\label{051}
\langle [\xi, \eta], \zeta] \rangle + \langle \eta,[\xi,\zeta]\rangle = 0
\end{equation}

We are assuming that $\langle \cdot, \cdot \rangle$ is nondegenerate on 
$[\n, \n]$.  Then $\n= [\n, \n] + \mathfrak{v}$.
If $[\n, \n]$ is positive or negative definite, then (\cite[Theorem 4.12]{CW})
$\n$ is commutative or $2$-step nilpotent.  Those cases aside,
we suppose that $[\n, \n]$ is indefinite.  

We are going to prove that $[\n,\n]$ is abelian.  Since $[\n,\n]$ is 
nilpotent, this is automatic if $\dim [\n,\n] < 3$, so we
may suppose $\dim [\n,\n] \geq 3$.  Write $\dim [\n, \n]=p+3$ with $p \geq 0$.
Fix $x \in \mathfrak{v}$.  By Lemma \ref{lemma4}, $[\n, \n]$ has a basis
$\{e_1, e_2, \cdots , e_{p+3}\}$ in which the inner product has matrix
\begin{equation} \label{052}
\langle \cdot, \cdot \rangle_{[\n, \n]} = \left ( \begin{smallmatrix}
                \begin{smallmatrix}
                 0 &0 &-1 \\ 0 &1 &0\\ -1 &0 &0 \end{smallmatrix}
                  & {\mathbf 0}\\
                  {\mathbf 0} & I_p
                \end{smallmatrix}\right )
\end{equation}
and if $\ad (x)\not=0$ it has matrix
\begin{equation}\label{5003}
              \ad (x)|_{[\n, \n]} = \left ( \begin{smallmatrix}
                \begin{smallmatrix}
                 0 &1 &0 \\ 0 &0 &1\\ 0 &0 &0 \end{smallmatrix}
                  & {\mathbf 0}\\
                  {\mathbf 0} &  {\mathbf 0}
                \end{smallmatrix}\right )
\end{equation}
If $\ad (x)=0$ for every $x \in \mathfrak{v}$, then $\n$ is 2-step nilpotent, 
because $\mathfrak{v}$ generates $\n$. In the following we assume that there 
exists $x \in \mathfrak{v}$ such that $\ad (x)\not=0$.
\smallskip

Let $y \in \mathfrak{v}$.  The matrix of $\ad (y)$ in 
the given basis on $[\n, \n]$ has form $\left ( \begin{smallmatrix}
A &B \\ C  &D \end{smallmatrix} \right )$
where $D\in {\mathbb R}^{p\times p}$. By Lemma \ref{lemma4}, in a possibly
modified basis, either
$\ad (y)|_{[\n, \n]} = \left ( \begin{smallmatrix}
                \begin{smallmatrix}
                 0 &1 &0 \\ 0 &0 &1\\ 0 &0 &0 \end{smallmatrix}
                  & {\mathbf 0}\\
                  {\mathbf 0} &  {\mathbf 0}
                \end{smallmatrix}\right )$
or ${\ad (y)}|_{[\n, \n]}= \mathbf 0$.
\smallskip

If $\ad (y)|_{[\n, \n]} \ne 0$ the rank $r(\ad (y)|_{[\n, \n]})=2$.
If $D \ne 0$ let $\lambda \gg 0$ be a large real number
and consider the matrix of
$\ad (\lambda x +y)|_{[\n, \n]}$.  For $\lambda$ it has rank $\ge 3$.
This contradiction shows $D = \mathbf 0$.
\smallskip

Express $A= \left ( \begin{smallmatrix}
                 a_{11} &a_{12} &a_{13} \\
                   a_{21} &a_{22} &a_{23} \\
                   a_{31} &a_{32} &a_{33}
                \end{smallmatrix} \right )$.
In the basis $\{e_1, e_2, \cdots , e_{p+3}\}$,
$\ad (\lambda x +y)|_{[\n, \n]}$ has matrix
$\left ( \begin{smallmatrix}
\tilde A & \tilde B \\ \tilde C  & \mathbf{0} \end{smallmatrix} \right )$
where $ \tilde A= \left ( \begin{smallmatrix}
                 a_{11} &a_{12}+\lambda &a_{13} \\
                   a_{21} &a_{22} &a_{23} +\lambda \\
                   a_{31} &a_{32} &a_{33}\\
                \end{smallmatrix} \right )$.
If $a_{31} \ne 0$ and  $\lambda \gg 0$ then $\det (\tilde A) \ne 0$,
so $r(\ad (\lambda x +y)|_{[\n, \n]}) \ge3$.  This contradiction
shows $a_{31} = 0$.  Similarly the first column of $C$ and the third row of
$B$ vanish.  We will need these constraints on the matrix
$\ad (y)|_{[\n, \n]} = \left ( \begin{smallmatrix}
                 A & B \\ C & {\mathbf 0} \end{smallmatrix} \right )$.
\smallskip

The metric on $N$ has matrix
$\left ( \begin{smallmatrix} W & {\mathbf 0}\\{\mathbf 0}  & I_p
\end{smallmatrix} \right )$ where $W= \left ( \begin{smallmatrix}
                 0 &0 &-1 & \\
                 0 &1 &0  \\
                 -1 &0 &0 \end{smallmatrix} \right )$,
by equation (\ref{051}).  Thus
$$ \left ( \begin{smallmatrix} A^t &C^t \\
                     B^t  & {\mathbf 0} \end{smallmatrix} \right )
	\left ( \begin{smallmatrix} W & {\mathbf 0} \\ {\mathbf 0} & I_p
		\end{smallmatrix} \right ) +
	\left ( \begin{smallmatrix} W & {\mathbf 0} \\ {\mathbf 0} & I_p
                \end{smallmatrix} \right ) \left ( \begin{smallmatrix}
		A &B \\
                     C  & {\mathbf 0} \end{smallmatrix} \right )=\mathbf 0.$$
In other words,
\begin{equation}\label{502}
A^t\, W= -W A \text{ and } C^t=-W B.
\end{equation}

Let $A= \left ( \begin{smallmatrix}
                 a_{11} &a_{12} &a_{13} \\
                   a_{21} &a_{22} &a_{23} \\
                   0 &a_{32} &a_{33}
                \end{smallmatrix} \right )$,
                $B= \left ( \begin{smallmatrix}
                 b_{11} & b_{12}&\cdots &b_{1p} \\
                   b_{21} &b_{22} &\cdots &b_{2p} \\
                   0 &0 &\cdots &0
                \end{smallmatrix} \right )$ and
		$C^t= \left ( \begin{smallmatrix}
                   0 &0 &\cdots &0 \\
                   c_{21} &c_{22}&\cdots &c_{2p} \\
                    c_{31} &c_{32} &\cdots &c_{3p}
                \end{smallmatrix} \right )$.
By (\ref{502}), $A= \left ( \begin{smallmatrix}
                 a_{11} &a_{12} &0 \\
                   a_{21} &0 &a_{12} \\
                   0 &a_{21} &-a_{11}
                \end{smallmatrix} \right )$.
Further. $c_{2i}=-b_{2i}$ and $c_{3i}=b_{1i}$ for $i= 1, \cdots, p$.  
It follows that
\begin{equation}\label{5004}
\ad (y)|_{[\n, \n]}=
\left(\begin{array}{ccc:cccc}
               a_{11} &a_{12} &0 &b_{11} & b_{12}&\cdots &b_{1p}\\
               a_{21} &0 &a_{12} &b_{21} &b_{22} &\cdots &b_{2p} \\
               0 &a_{21} &-a_{11} &0 &0 &\cdots &0\\
                \hdashline
                 0 & -b_{21} & b_{11} & & & &\\
                 \vdots & \vdots & \vdots & & \mathbf{0} & &\\
                  0 & -b_{2p} & b_{1p} & & & &\\
                \end{array}\right).
\end{equation}
Since $\ad(\n)$ preserves both $[\n,\n]$ and $\mathfrak{v} = [\n,\n]^\perp$,
we have $\ad ([x, y])|_{[\n,\n]} = [\ad(x)|_{[\n,\n]}, \ad(y)|_{[\n,\n]}]$.
Combining (\ref{5003}) and (\ref{5004}), now,
\begin{equation*}
\ad ([x, y]) |_{[\n, \n]}=
              \left(\begin{array}{ccc:cccc}
               a_{21} &-a_{11} &0 &b_{21} &b_{22} &\cdots &b_{2p}\\
               0 &0 &-a_{11} &0 &0 &\cdots &0 \\
               0 &0 &a_{21} &0 &0 &\cdots &0\\
                \hdashline
                 0 & 0 & b_{21} & & & &\\
                 \vdots &\vdots & \vdots  & & \mathbf{0} & &\\
                  0 & 0& b_{2p}  & & & &\\
                \end{array}\right).
\end{equation*}

As $\ad ([x, y])|_{[\n, \n]}$ is nilpotent, its eigenvalues all are zero.
Thus $a_{21}=0$.  Similarly, from (\ref{5004}), $a_{11}=0$. Moreover,
the matrix of $(\ad (y)|_{[\n, \n]})^2$ is
\begin{small}
\begin{equation}\label{5005}
(\ad (y)|_{[\n,\n]})^2=
\left(\begin{array}{ccc:cccc}
              0 &-\sum^p_{i=1}b_{1i}b_{2i} &
		a_{12}^2+\sum^p_{i=1}b_{1i}^2 &b_{21} &
			b_{22}&\cdots &b_{2p}\\
              0 &-\sum^p_{i=1}b_{2i}^2 &
		\sum^p_{i=1}b_{1i}b_{2i}&0&0&\cdots &0 \\
              0 &0 &0 &0 &0 &\cdots &0\\
         \hdashline
              0 &0  & -a_{12}b_{21} &
		-b_{21}^2 &-b_{21}b_{22} &\cdots &-b_{21}b_{2p}\\
                 \vdots &\vdots & \vdots  & \vdots & \vdots &\vdots & \vdots \\
              0 & 0 & -a_{12}b_{2p} &-b_{2p}b_{21} &-b_{2p}b_{22}  &
		\cdots &-b_{2p}^2\\
              \end{array}\right).
\end{equation}
\end{small}
From that we compute the trace Tr($(\ad (y)|_{[\n, \n]})^2) =
-2\sum_1^p b_{2i}^2$.  As $(\ad (y)|_{[\n, \n]})^2$ is nilpotent, it has
trace $0$, so $b_{21}= \cdots =b_{2p}=0$.
\smallskip

From these calculations we have
\begin{equation}\label{vanish1}
\ad ([x, y]) |_{[\n, \n]}=0 \text{ for all } y \in \mathfrak{v}.
\end{equation}
Writing $a(y)$ for $a_{12}$ and $b_j(y)$ for $b_{1j}$ we also have
\begin{equation}\label{firstapprox}
\ad (y)|_{[\n, \n]}=
\left(\begin{array}{ccc:cccc}
               0 &a(y) &0 &b_{1}(y)& b_{2}(y)&\cdots &b_{p}(y)\\
               0 &0 &a(y) &0 &0 &\cdots &0 \\
               0 &0 &0 &0 &0 &\cdots &0\\
                \hdashline
                 0 & 0 & b_{1}(y) & & & &\\
                 \vdots &\vdots & \vdots & & \mathbf{0} & &\\
                  0 & 0 & b_{p}(y) & & & &\\
                \end{array}\right), \quad \text{ for all } y \in \mathfrak{v}.
\end{equation}
Initially (\ref{firstapprox}) requires $y$ to be linearly independent of $x$,
but it
holds for all $y \in \mathfrak{v}$ with $a(y) = 1$ and $b_j(y) = 0$.
\smallskip

We continue to simplify the structure of $\ad (y)|_{[\n, \n]}$.  For the
moment assume
$\dim\,\mathfrak{v} = s+1 \geq 2$.  Extend $\{x\}$ to a basis
$\{x, x_1,\cdots, x_s \}$ of $\mathfrak{v}$.  Using (\ref{firstapprox})
\begin{equation*}
{\ad (x_i)}|_{[\n, \n]}=
              \left(\begin{array}{ccc:cccc}
               0 &a(x_i) &0 &b_1(x_i) & b_2(x_i)&\cdots &b_p(x_i)\\
               0 &0 &a(x_i) &0 &0 &\cdots &0 \\
               0 &0 &0 &0 &0 &\cdots &0\\
                \hdashline
                 0 & 0 & b_1(x_i) & & & &\\
                 \vdots &\vdots & \vdots & & \mathbf{0} & &\\
                  0 & 0 & b_p(x_i) & & & &\\
                \end{array}\right) \text{ for } 1 \leq i \leq s.
\end{equation*}

From this point on, in the proof of Theorem \ref{th41}, we will make
successive modifications of the basis $\{x, x_1, \cdots, x_s-a(x_s) x \}$,
along the lines of Gauss Elimination.  To avoid complicated notation
we use $\{x, \tilde x_1,\cdots, \tilde x_s \}$ for each of
the successive modifications.
\smallskip

In the basis $\{x, \tilde x_1,\cdots, \tilde x_s \}:=
\{x, x_1-a(x_1) x,\cdots, x_s-a(x_s) x \}$ we now have
\begin{equation}\label{tr1}
{\ad (\tilde x_i)}|_{[\n, \n]}=
           \left(\begin{array}{ccc:cccc}
            0 &0 &0 &b_1(\tilde x_i) & b_2(\tilde x_i)&\cdots &b_p(\tilde x_i)\\
               0 &0 &0 &0 &0 &\cdots &0 \\
               0 &0 &0 &0 &0 &\cdots &0\\
                \hdashline
                 0 & 0 & b_1(\tilde x_i) & & & &\\
                 \vdots &\vdots &\vdots  & & \mathbf{0} & &\\
                  0 & 0 & b_p(\tilde x_i) & & & &\\
                \end{array}\right) \text{ for } 1 \leq i \leq s.
\end{equation}
From (\ref{tr1}) we compute $[\ad(\tilde x_i),\ad(\tilde x_j)]|_{[\n,\n]}
= [\ad(\tilde x_i)|_{[\n,\n]},\ad(\tilde x_j)|_{[\n,\n]}] = 0$ for
$1 \le i,j \le s$.  Also, from (\ref{5003}) together with
(\ref{tr1}), $[\ad(x),\ad(\tilde x_j)]|_{[\n,\n]}
= [\ad(x)|_{[\n,\n]},\ad(\tilde x_j)|_{[\n,\n]}] = 0.$  Thus
\begin{equation}\label{vn}
[[\mathfrak{v},\mathfrak{v}], [\n,\n]] = 0.
\end{equation}

Assume $s \geq 2$.  Write $[x, \tilde x_i]=\sum_{k=1}^{p+3} a_i^k e_k$
and $[\tilde x_i, \tilde x_j] = \sum_{k=1}^{p+3} a_{i,j}^k e_k$\,.  Then
\begin{equation}\label{jacobi}
\begin{aligned}
&[\tilde  x_i,[x, \tilde x_j]] = {\sum}_{k=1}^{p+3} a_j^k[\tilde x_i, e_k]
= a_j^3 {\sum}_{\ell=1}^pb_\ell (\tilde x_i) e_{\ell+3} +
\left ( {\sum}_{k=1}^p a_j^{k+3}b_k(\tilde x_i)\right ) e_1\,, \\
&[x, [\tilde x_j , \tilde  x_i]] =
        {\sum}_{k=1}^{p+3} a_{j,i}^k[x,e_k] = a_{j,i}^2\,e_1 + a_{j,i}^3\,e_2\,,
        \text{ and } \\
&[\tilde x_j, [\tilde x_i,x]] = -{\sum}_{k=1}^{p+3} a_i^k[\tilde x_j, e_k]
        = -a_i^3 {\sum}_{\ell=1}^p b_\ell(\tilde x_j) e_{\ell +3} +
\left ( {\sum}_{k=1}^p a_i^{k+3}b_k(\tilde x_j)\right ) e_1
\end{aligned}
\end{equation}
The first and third terms here have no $e_2$ component.  From the Jacobi
Identity $[x, [\tilde x_j , \tilde  x_i]]$ has no $e_2$ component, i.e.
$a_{j,i}^3 = 0$,
so $[x, [\tilde x_j , \tilde  x_i]] = a_{j,i}^2\,e_1$ and
$[\tilde x_j , \tilde  x_i]$ has no $e_3$ component.
\smallskip

Suppose that $[x,\tilde x_j]$ has nonzero $e_3$ component.  At least one
of those $e_3$ components is nonzero because $\mathfrak{v}$ generates $\n$.
We next modify the basis $\{x,\tilde x_1, \dots , \tilde x_s\}$ of
$\mathfrak{v}$ by (1) permuting the $\{\tilde x_j\}$ if necessary so that
$[x, \tilde x_1]$ has nonzero $e_3$ component, and (2) if $j > 1$ and
$[x, \tilde x_j]$ has nonzero $e_3$ component then subtract a multiple of
$\tilde x_1$ from $\tilde x_j$ so that $[x,\tilde x_j]$ has $e_3$ component
zero.  Then (\ref{tr1}), and thus (\ref{jacobi}), still hold for
the modified $\tilde x_i$\,.
\smallskip

We have arranged $a_{j,i}^3$ = 0 for $1 \leq i,j \leq s$, $a_1^3 \ne 0$, and
$a_k^3 = 0$ for $k > 1$.  Thus, in (\ref{jacobi}),
$[\tilde  x_i,[x, \tilde x_j]]$ is a multiple of $e_1$ when $j > 1$,
$[x, [\tilde x_j , \tilde  x_i]]$ is a multiple of $e_1$ in general, and
$[\tilde x_j, [\tilde x_i,x]]$ is a multiple of $e_1$ when $i > 1$.
From the Jacobi Identity, if $i > 1$ then $[\tilde  x_i,[x, \tilde x_1]]$ is
a multiple of $e_1$\,.  Again from (\ref{jacobi})
$a_1^3 {\sum}_{\ell=1}^pb_\ell (\tilde x_i) e_{\ell+3} = 0$, and since
$a_1^3 \ne 0$ this says that each $b_\ell (\tilde x_i) = 0$.  Going back to
(\ref{tr1}),
$$
\text{ if } i > 1 \text{ then } \ad(\tilde x_i)|_{[\n,\n]} = 0.
$$

In summary we see that $\n$ has a very simple structure.
Associated with an appropriate basis $\{x, \tilde x_1,\cdots, \tilde x_s \}$
of $\mathfrak{v}$ $(s\ge 2)$, 
\begin{equation}\label{structure1}
\begin{aligned}
&{\ad (x)}|_{[\n, \n]}=
               \left ( \begin{matrix}
                \begin{smallmatrix}
                 0 &1 &0 \\ 0 &0 &1\\ 0 &0 &0 \end{smallmatrix}
                  & {\mathbf 0}\\
                  {\mathbf 0} &
                  {\mathbf 0}
                \end{matrix} \right ) ; \qquad
\ad (\tilde x_1)|_{[\n, \n]}=
              \left(\begin{matrix} {\mathbf 0} &
                \begin{smallmatrix}
               a_1 & a_2&\cdots &a_p\\
               0 &0 &\cdots &0 \\
               0 &0 &\cdots &0 \end{smallmatrix} \\
                \begin{smallmatrix} 0 & 0 & a_1 \\
                        \vdots & \vdots & \vdots \\
                        0 & 0 & a_p \end{smallmatrix} & {\mathbf 0}
                \end{matrix} \right ); \\
& \ad (\tilde x_i)|_{[\n, \n]}=0 \text{ for } 2 \leq i \leq s; \text{ and }
\ad ([y, z])|_{[\n, \n]}=0 \text{ for all } y, z \in \n.
\end{aligned}
\end{equation}
In particular $[\n, \n]$ is abelian.  Thus $\n$ is at most $4$-step nilpotent.

In order to see that $\n$ cannot be $3$ step nilpotent we use the structure
just described. Suppose $[\n, \n] \ne 0 \ne [\n,[\n, \n]]$.  Then, by 
construction, $[\n, \n]=Span\{e_1, ...,e_{p+3}\}$.  
By (\ref{structure1}), if $p=0$, $[\n,[\n, \n]]=Span\{e_1, e_2\}$; if $p \ne 0$, $[\n,[\n, \n]]=Span\{e_1, e_2, \sum_{i=1}^p a_{i}e_{3+i}\}$.
Note that $e_3$ is not contained in this span.
Continuing with (\ref{structure1})
we see $[\n,[\n,[\n, \n]]]=Span\{e_1\}$.  This eliminates the possibility of
$3$--step nilpotence and completes the proof of Theorem \ref{th41}.
\end{proof}

\begin{remark}\label{4step}
In connection with Theorem \ref{th41}, there are many examples where
$\n$ is abelian or $2$--step nilpotent, but we have not been able
to construct an example where it is $4$--step nilpotent.  So, at the 
moment, $4$--step nilpotence is only a necessary condition.
\end{remark}

\section{Lorentz Geodesic Orbit and Weakly Symmetric Nilmanifolds, II}
\label{sec5}

The second of our two main results, the case where
the metric is degenerate on $[\n, \n]$, is as follows.
The result contrasts with Theorem \ref{th41}, and essentially coincides
with the situation for Riemannian manifolds.

\begin{theorem}\label{th42}
Let $(M=G/H, \langle, \rangle)$ be a connected Lorentz geodesic orbit
nilmanifold.  
Suppose that $G =N \rtimes H$ with $N$ nilpotent.
Suppose further that there is a reductive decomposition
$\g=\n + \h$, where $[\n, \n]$ is degenerate and the action of
$\Ad(H)|_\n$ is completely reducible on $\n$. Then $\n$ is at most
$2$-step nilpotent.
\smallskip

Furthermore, there is a basis $\{e_1, \cdots, e_p; e_{p+1}\}$ of
$[\n,\n]$ and a basis $\{v_0; v_1, \cdots , v_s\}$ of a vector space complement
$\mathfrak{a}$ to $[\n,\n]$ in $\n$ with the following properties.

{\rm (1)} $\mathfrak{v_1} := {\rm Span}(e_1, \cdots, e_p)$ and
$\mathfrak{v_2} := {\rm Span}(v_1, \cdots , v_s)$ are both positive definite
or both negative definite and are $\Ad(H)$--invariant,

{\rm (2)} $[\n,\n] \cap [\n,\n]^\perp =
e_{p+1}\mathbb{R}$ and $\mathfrak{a} \cap \mathfrak{a}^\perp = v_0\mathbb{R}$
are $\Ad(H)$--invariant,

{\rm (3)} $\mathfrak{w} := {\rm Span}(e_{p+1}, v_0)$ is of signature $(1,1)$,

{\rm (4)} $\n = \mathfrak{v_1} + \mathfrak{w} + \mathfrak{v_2}$
is an $\Ad(H)$--invariant orthogonal direct sum,

{\rm (5)} ${\ad (x)}|_{[\n, \n]}=0$ for any $x \in \mathfrak{a}$.
\end{theorem}

\begin{proof}

Let $ \dim [\n, \n]=p+1$.  Since $\n$ is of Lorentz signature and
$[\n, \n]$ is degenerate, $\dim([\n, \n] \cap [\n, \n]^\perp) = 1$.
So we have $e_{p+1} \ne 0$ spanning $[\n, \n] \cap [\n, \n]^\perp$,
and $e_{p+1}^\perp = \mathfrak{v} + e_{p+1}\mathbb{R}$ where
$\mathfrak{v}$ is positive or negative definite.  Now $\mathfrak{v}
= \mathfrak{v}_1 + \mathfrak{v}_2$\,, $\Ad(H)$--invariant orthogonal
direct sum, where $\mathfrak{v}_1 = \mathfrak{v} \cap [\n, \n]$.
Thus $\mathfrak{w} := \mathfrak{v}^\perp$ is spanned by
$e_{p+1}$ and a null vector $v_0$ with $\langle e_{p+1},v_0\rangle = 1$
and $\Ad(H)v_0 \in v_0\mathbb{R}$.  Choose orthonormal bases
$\{e_1, \cdots , e_p\}$ of $\mathfrak{v_1}$ and $\{v_1, \cdots , v_s\}$
of $\mathfrak{v_2}$\,.  With those, we have constructed a basis of $\n$ that
satisfies conditions (1) through (4) above.
Note that the inner product on $\mathfrak{w}$ has matrix
$\left ( \begin{smallmatrix} 0 & 1 \\ 1 & 0 \end{smallmatrix} \right )$. So the metric on $\n$ has the matrix
\begin{equation*}
\langle \cdot, \cdot \rangle_{\n}=\left (\begin{smallmatrix}
                  I_p &{\mathbf 0} &{\mathbf 0} &{\mathbf 0}\\
                    {\mathbf 0} &0  &1  &{\mathbf 0} \\
                     {\mathbf 0} &1  &0  &{\mathbf 0} \\
                      {\mathbf 0} &{\mathbf 0}  &{\mathbf 0}  &I_s
                \end{smallmatrix} \right )
\end{equation*}
under the basis $\{e_1, e_2, \cdots , e_{p}; e_{p+1}, v_0; v_1, \cdots, v_s\}$. In particular, the metric on $[\n, \n]$ has the matrix $\langle \cdot, \cdot \rangle_{[\n, \n]}=\left (\begin{smallmatrix}
                  I_p & {\mathbf 0} \\
                    {\mathbf 0} &0
                \end{smallmatrix}\right )$ which is degenerate.
\smallskip

Let $x \in \n$.  Then $\ad(x)$ preserves $[\n,\n] = \mathfrak{v}_1 +
e_{p+1}\mathbb{R}$ and $\ad(x)|_{[\n,\n]}$
has matrix, relative to $\{e_1,\cdots,e_p;e_{p+1}\}$, of the form
$\left ( \begin{smallmatrix} A & B \\ C & d \end{smallmatrix} \right ) =
\left ( \begin{smallmatrix} A(x) & B(x) \\ C(x) & d(x) \end{smallmatrix}
\right ).$
\smallskip

First we consider $\ad(v_0)|_{[n,n]} = \left ( \begin{smallmatrix} A(v_0)
        & B(v_0) \\ C(v_0) & d(v_0) \end{smallmatrix} \right )$.
By the Geodesic Lemma, there exists $a_{v_0} \in \mathfrak{h}$
such that $\langle [v_0 +a_{v_0}, e_i], v_0 \rangle = k\langle v_0, e_i
\rangle=0$ for $1 \le i \le p.$
  Since $H$ is completely reducible on
$\mathfrak{g}$ we have $[a_{v_0}, e_i] \in \mathfrak{v}_1$\,.  Now
$\langle [a_{v_0}, e_i] , v_0\rangle=0$,
so $$\langle [v_0, e_i], v_0\rangle=0.$$ Now $[v_0, e_i]=
C_i(v_0)e_{p+1}+\sum_{j=1}^p a_{ji}e_j$ for any $1\leq i\leq p$.
So $$\langle [v_0, e_i], v_0\rangle=\langle C_i(v_0)e_{p+1}, v_0\rangle
=C_i(v_0), \quad 1\leq i\leq p.$$
It forces $C_i(v_0) =0$, i.e. $C(v_0) = \mathbf{0}$. Now $\ad(v_0)|_{[n,n]} = \left ( \begin{smallmatrix} A(v_0)
        & B(v_0) \\ {\mathbf 0} & d(v_0) \end{smallmatrix} \right )$.
Furthermore, for any $e\in \mathfrak{v_1}$, by the Geodesic Lemma, there exists
$a_e \in \mathfrak{h}$ such that
$
\langle [e +a_e, v_0], e \rangle = k\langle v_0, e \rangle.
$
Since $v_0$ is a one dimensional submodule, we know $\langle [a_e, v_0], e \rangle=0$. Hence
$
\langle [e, v_0], e \rangle=0$. It follows that $$\langle [v_0,e_i],e_j\rangle+\langle e_i, [v_0,e_j]\rangle=0, \quad 1 \le i,j \le p.$$
Then we have $a_{ij}+a_{ji}=0$,
that says $A(v_0)^t=-A(v_0)$. Since $\ad(v_0)$ is nilpotent, we have
$A(v_0)=\mathbf{0}$ and $d(v_0)=0$. Thus,
\begin{equation*}
\ad (v_0)|_{[\n,\n]}=  \left ( \begin{smallmatrix} \mathbf{0} &B(v_0) \\
                     \mathbf{0}  &0  \end{smallmatrix} \right ).
\end{equation*}

Next consider $\ad(v)|_{[n,n]} = \left ( \begin{smallmatrix} A(v)
        & B(v) \\ C(v) & d(v) \end{smallmatrix} \right )$ for any $v \in \mathfrak{v}_2$. We write equation (\ref{051}) as
$$\left ( \begin{smallmatrix} A(v)^t &C(v)^t \\ B(v)^t  &d \end{smallmatrix} \right )
\left ( \begin{smallmatrix} I_p & \mathbf{0} \\
		\mathbf{0}  & 0 \end{smallmatrix} \right ) +
\left ( \begin{smallmatrix} I_p & \mathbf{0} \\
		\mathbf{0}  & 0 \end{smallmatrix} \right )
\left ( \begin{smallmatrix} A(v) &B(v) \\ C(v)  &d \end{smallmatrix} \right )=0.$$
It follows that $A(v)^t = -A(v)$ and $B(v) = \mathbf{0}$.
Since $\ad (v)|_{[\n, \n]}$ is nilpotent, now we have $A(v)=\mathbf{0}$ and $d=0$, so
$\ad (v)|_{[\n, \n]}=
\left ( \begin{smallmatrix} \mathbf{0} & \mathbf{0} \\
                     C(v)  & 0  \end{smallmatrix} \right ).$
We now apply the Geodesic Lemma to $\ad(v+v_0)|_{[\n,\n]}$.
That gives us $a_{v+v_0} \in \mathfrak{h}$ such that
$\langle [v + v_0 + a_{v+v_0}, e_i],v + v_0 \rangle =
k\langle v + v_0,e_i\rangle = 0$ for $1\le i \le p.$
  Since
$[a_{v+v_0},e_i] \in [\mathfrak{h},\mathfrak{v}_1] \subset \mathfrak{v}_1$\,, it follows that
$$\langle [v+v_0,e_i],v+v_0\rangle = 0.$$ On the other hand, since $[v_0,e_i]=0$ for any $1\le i \le p$, we have
\begin{eqnarray*}
\langle [v+v_0,e_i],v+v_0\rangle & = & \langle [v,e_i],v+v_0\rangle = \langle C_i(v)e_{p+1},v+v_0\rangle \\
& = & \langle C_i(v)e_{p+1},v_0\rangle \\
&= & C_i(v).
\end{eqnarray*}
This forces $C_i(v) =0$, i.e. $C(v) = \mathbf{0}$. Thus
$\ad(v)|_{[\n,\n]} = 0$.
\smallskip

Furthermore for any $v \in \mathfrak{v}_2$, first we get 
$\ad (v)|_{\mathfrak{v_1} + \mathfrak{w}}=  
\left ( \begin{smallmatrix} \mathbf{0} & B_1(v) \\
                     \mathbf{0}  &0  \end{smallmatrix} \right )$
in the basis $\{e_1,\cdots,e_{p+1},v_0\}$ since $\ad (v)|_{[\n,\n]}= 0$.
By the Geodesic Lemma to $\ad(v)|_{\mathfrak{v_1} + \mathfrak{w}}$ and the 
fact that $\mathfrak{v}_2$ and $\mathfrak{v_1} + \mathfrak{w}$ are 
$\Ad(H)$-invariant,  we know 
$\ad (v)|_{\mathfrak{v_1} + \mathfrak{w}}\in \so(p+1, 1)$. 
By Lemma \ref{lemma4}, we have $\ad (v)|_{\mathfrak{v_1} + \mathfrak{w}}= 0$. 
That is, $[v,v_0]=0$ for any $v\in \mathfrak{v}_2$, then for any $v\in\mathfrak a$.
\smallskip

Since $\mathfrak a$ generates $\n$, $\ad(v)|_{[\n,\n]} = 0$ for any $v \in \mathfrak{v}_2$, and $
\ad (v_0)|_{[\n,\n]}=  \left ( \begin{smallmatrix} \mathbf{0} &B(v_0) \\
                     \mathbf{0}  &0  \end{smallmatrix} \right )$, there exist $y, z \in \mathfrak a$ such that $[y, z]=\sum^{p+1}_{i=1} a_i e_i$ with $a_{p+1}\ne 0$. Then $$[v_0, [y, z]]=a_{p+1}[v_0,e_{p+1}]=a_{p+1}\sum_{i=1}^pB_i(v_0)e_i.$$ From the Jacobi Identity and the fact $[v,v_0]=0$ for any $v\in\mathfrak a$, we have
                $$[v_0, [y, z]]=[[v_0, y], z]+[y, [v_0, z]]=0,$$
it forces $B(v_0)=\mathbf{0}$. That is $\ad (v_0)|_{[\n, \n]}= 0$.
\smallskip

Now we know $\ad (x)|_{[\n, \n]}= 0$ for any $x\in \mathfrak a$, and thus
also for any $x \in \n$. Thus $\n$ is at most $2$-step nilpotent since
$\mathfrak a$ generates $\n$.
\end{proof}

\begin{remark}\label{wk-symm-case}
By Proposition \ref{propnat}, Theorem \ref{th41} holds in particular 
where $(M,\langle\,,\,\rangle)$ is a naturally reductive Lorentz nilmanifold.  
By Proposition \ref{prop41}, Theorem \ref{th42} holds in particular
where $(M,\langle\,,\,\rangle)$ is a connected  weakly symmetric Lorentz 
nilmanifold with $G ={I(M)}^0${\rm )}.
\hfill $\diamondsuit$
\end{remark}

\section{Acknowledgements} ZC was partially supported by National Natural 
Science Foundation of China (11931009 and 12131012) and Natural Science Foundation of 
Tianjin (19JCYBJC30600). JAW was partially supported by the Simons Foundation.
SZ was partially supported by National Natural Science Foundation of China 
(12071228 and 51535008), and SZ thanks the China Scholarship Council for 
support at University of California, Berkeley and he also thanks U. C. 
Berkeley for hospitality.

\end{document}